\documentclass[11pt]{amsart}

\makeatletter
\usepackage{amssymb}
\usepackage{latexsym}
\usepackage{amsbsy}
\usepackage{amsfonts}

\def\marginpar#1{\ignorespaces}

\newtheorem{theorem}{Theorem}

\newtheorem{lemma}[theorem]{Lemma}

\theoremstyle{definition}
\newtheorem{remark}[theorem]{Remark}

\numberwithin{equation}{section}
\numberwithin{theorem}{section}

\def\AArm{\fam0 \rm}%
\newdimen\AAdi%
\newbox\AAbo%
\def\AAk#1#2{\setbox\AAbo=\hbox{#2}\AAdi=\wd\AAbo\kern#1\AAdi{}}%

\newcommand{\BBone}{{\ensuremath{{\AArm 1\AAk{-.8}{I}I}}}}

\def\eqref#1{(\ref{#1})}
\def\eqlabel#1{\def\@currentlabel{#1}}

\def\formula#1{\def\@tempa{#1}\let\@tempb\theequation\def\theequation{%
\hbox{#1}}\def\@currentlabel{(\theequation)}$$}
\def\endformula{\leqno\hbox{(\@tempa)}$$\@ignoretrue\let\theequation\@tempb}

\def\given{\hskip5\p@\relax\vrule\@width.4\p@\hskip5\p@\relax}

\newcommand{\open}[1]{%
\par\normalfont\topsep6\p@\@plus6\p@\trivlist\item[\hskip\labelsep\itshape#1%
\@addpunct{.}]\ignorespaces}

\DeclareRobustCommand{\close}[1]{%
  \ifmmode 
  \else \leavevmode\unskip\penalty9999 \hbox{}\nobreak\hfill
  \fi
  \quad\hbox{$#1$}}

\newlength{\toskip}

\def\<{\langle}
\def\>{\rangle}

\def \R {{\mathbb R}}

\def \E {{\mathbb E}}

\def \L {{\mathbb L}}

\def \Ent {\textrm{Ent}}

\makeatother

\begin{document}
\date{\today}

\title[Functional Inequalities via Lyapunov conditions]{Functional Inequalities via Lyapunov conditions}

 \author[P. Cattiaux]{\textbf{\quad {Patrick} Cattiaux $^{\spadesuit}$ \, \, }}
\address{{\bf {Patrick} CATTIAUX},\\ Institut de Math\'ematiques de Toulouse. CNRS UMR 5219. \\
Universit\'e Paul Sabatier,
\\ 118 route
de Narbonne, F-31062 Toulouse cedex 09.} \email{cattiaux@math.univ-toulouse.fr}

 \author[A. Guillin]{\textbf{\quad {Arnaud} Guillin $^{\diamondsuit}$}}
\address{{\bf {Arnaud} GUILLIN},\\ Laboratoire de Math\'ematiques, CNRS UMR 6620, Universit\'e Blaise Pascal, avenue des Landais 63177 Aubi\`ere.}
\email{guillin@math.univ-bpclermont.fr}

\maketitle
 \begin{center}

 \textsc{$^{\spadesuit}$  Universit\'e de Toulouse}
\smallskip

\textsc{$^{\diamondsuit}$ Universit\'e Blaise Pascal}
\smallskip

 \end{center}

\begin{abstract}
We review here some recent results by the authors, and various coauthors, on (weak,super)
Poincar\'e inequalities, transportation-information inequalities or logarithmic Sobolev inequality
via a quite simple and efficient technique: Lyapunov conditions. \end{abstract}
\bigskip

\bigskip

\bigskip

\section{Introduction and main concepts}

Lyapunov conditions appeared a long time ago.
They were particularly well fitted to deal with the problem of convergence
to equilibrium for Markov processes, see \cite{MT,MT2,MT3,DFG} and references therein.
They also appeared earlier in the study of large and moderate deviations for empirical functionals of Markov processes
(see for examples Donsker-Varadhan \cite{DV3,DV4}, Kontoyaniis-Meyn \cite{KM1,KM2}, Wu \cite{wu1,wu2}, Guillin \cite{G1,G2},...),
for solving Poisson equation \cite{GM},...\\

Their use to obtain functional inequalities is however quite recent, even if one may afterwards
find hint of such an approach in Deuschel-Stroock  \cite{DS} or Kusuocka-Stroock \cite{KS85}. The
present authors and coauthors have developed a methodology that has been successful for various
inequalities: Lyapunov-Poincar\'e inequalities \cite{BCG}, Poincar\'e inequalities \cite{BBCG},
transportation inequalities for Kullback information \cite{CGW} or Fisher information \cite{GLWY},
Super Poincar\'e inequalities \cite{CGWW}, weighted and weak Poincar\'e inequalities \cite{CGGR}
or the forthcoming \cite{CGW2} for Super weighted Poincar\'e inequalities. We finally refer to the
forthcoming book \cite{CGbook} for a complete review. For more references on the various
inequalities introduced here we refer to \cite{bakry,logsob,ledoux01,Wbook}. The goal of this
short review is to explain the methodology used in these papers and to present various general
sets of conditions for this panel of functional inequalities. The proofs will of course be only
schemed and we will refer to the original papers for complete statements.

Let us first describe the framework of our study. Let $E$ be some Polish state space, $\mu$ a
probability measure, and a $\mu$-symmetric operator $L$.  The main assumption on $L$ is that there
exists some algebra $\mathcal{A}$ of bounded functions, containing constant functions, which is
everywhere dense (in the $\L_2(\mu)$ norm) in the domain of $L$. It enables us to define a
``carr\'e du champ'' $\Gamma$, {\it i.e.} for $f, g \in \mathcal{A}$, $L(fg)=f Lg + g Lf + 2
\Gamma(f,g)$. We will also suppose that $\Gamma$ is a derivation (in each component), {\it i.e.}
$\Gamma(fg,h)=f\Gamma(g,h) + g \Gamma(f,h)$, i.e. we are in the standard ``diffusion'' case in
\cite{bakry} and we refer to the introduction of \cite{cat4} for more details. For simplicity we
set $\Gamma(f)=\Gamma(f,f)$. Also, since $L$ is a diffusion, we have the following chain rule
formula $\Gamma(\Psi(f), \Phi(g))=\Psi'(f) \Phi'(g) \Gamma(f,g)$.

In particular if $E=\R^n$, $\mu(dx)= e^{-V(x)} dx$, where $V$ is smooth and $L=\Delta - \nabla
V.\nabla$, we may consider the $C^\infty$ functions with compact support (plus the constant
functions) as the interesting subalgebra $\mathcal{A}$, and then $\Gamma(f,g)=\nabla f \cdot
\nabla g$. It will be our main object.
\medskip

Now we define the notion of $\phi$-Lyapunov function. Let $W\geq 1$ be a smooth enough function on
$E$ and $\phi$ be a $\mathcal{C}^1$ positive increasing function defined on $\R^+$. We say that
$W$ is a $\phi$-Lyapunov function if there is an increasing family of exhausting sets
$(A_r)_{r\ge0} \subset E$  and some $b \geq 0$ such that for some $r_0>0$
\begin{equation}
LW \, \le \,  -\phi(W) \, + \, b \, \BBone_{A_{r_0}} \, .\label{lyap}
\end{equation}
 One has very different behavior depending on $\phi$: if $\phi$ is linear then results of \cite{MT2,MT3} asserts
that the associated semigroup converges to equilibrium with exponential speed (in total variation or with some weighted norm)
 so that it is legitimate to hope for a Poincar\'e inequality to be valid.

 When $\phi$ is superlinear
 (or more generally in the form $\phi\times W$ where $\phi$ tends to infinity )
 we may hope for stronger inequalities (Super Poincar\'e, ultracontractivity...).

 Finally if $\phi$ is sublinear, as asserted in \cite{DFG}, only subexponential convergence to equilibrium is valid,
 so we should be in the regime of weak Poincar\'e inequalities.   We will see class of examples in the next section.\\
 Note however that the result of \cite{MT2,MT3,DFG} are valid even in a fully degenerate hypoelliptic setting
 (kinetic Fokker-Planck equation for example (see \cite{wu1,BCG})) whereas we cannot hope for a (weak, normal or super)
 Poincar\'e inequality, due to the "degeneracy" of the Dirichlet form.

 We thus have to impose another condition which will often be a "local inequality" such as local Poincar\'e inequality
 (i.e. Poincar\'e inequality restricted to a ball, or a particular set) or local super Poincar\'e inequality,
 preventing degeneracy case but quite easy to verify for general locally bounded measures. \\

We will present now the main lemma which will show how the Lyapunov condition is used in our setting.

\begin{lemma} \label{lem:philyap}
 Let $\psi : \mathbb{R}^+ \to \mathbb{R}^+$ be a $\mathcal{C}^1$ increasing function.
 Then, for any $f\in \mathcal{A}$ and any positive $h\in D(\mathcal E)$,
$$
\int  \, \frac{-Lh}{\psi(h)} \, f^2 \, d\mu  \leq  \int \frac { \, \Gamma(f)}{ \psi'(h)} d\mu.
$$
In particular,
$$
\int  \, \frac{-Lh}{h} \, f^2 \, d\mu  \leq  \int \, \Gamma(f) d\mu.
$$
\end{lemma}

\begin{proof}
Since $L$ is $\mu$-symmetric, using that $\Gamma$ is a derivation and the chain rule formula, we
have
\begin{eqnarray*}
\int  \, \frac{-Lh}{\psi(h)} \, f^2 \, d\mu
& = & \int \, \Gamma \left( h, \frac{f^2}{\psi(h)} \right) \, d\mu
=
\int \left(\frac{2 \, f \,
\Gamma(f,h)}{\psi(h)} \, - \, \frac{f^2 \psi'(h) \Gamma(h)}{\psi^2(h)}\right) d\mu \, .
\end{eqnarray*}
Since $\psi$ is increasing and according to Cauchy-Schwarz inequality we get
\begin{eqnarray*}
\frac{f \, \Gamma(f,h)}{\psi(h)}
& \leq &
\frac{f \sqrt{\Gamma(f) \Gamma(h)}}{\psi(h)}
 =
\frac{ \sqrt{\Gamma(f)} }{ \sqrt{\psi'(h)} } \cdot
\frac{f \sqrt{\psi'(h) \Gamma(h)}}{\psi(h)}
\\
&\leq &
\frac{1}{2} \frac{\Gamma(f)}{\psi'(h)} + \frac{1}{2} \, \frac{f^2 \psi'(h)\,
\Gamma(h)}{\psi^2(h)}.
\end{eqnarray*}
The result follows.
\end{proof}

\begin{remark}
In fact the conclusion of the preceding Lemma, in the case where $\psi$ is the identity, holds in
a more general setting and requires only the reversibility assumption. It is thus valid for some
Markov jump case ($M/M/\infty$, Levy process,...), see \cite{GLWY} where the proof follows from a
large deviations argument, or on more general Riemanian manifolds.
\end{remark}

\section{Examples of Lyapunov conditions}
Before stating the results achievable by our method, let us present some examples of Lyapunov conditions.

We will restrict ourselves to the framework described before: $E=\R^n$, $\mu(dx)=Z e^{-V(x)} dx$
and $L=\Delta - \nabla V.\nabla$. The Lyapunov conditions may be quite different: first
because of the very nature of $V$ itself, secondly because of the choice of the Lyapunov function
$W$. Let us illustrate it in the Gaussian case $V(x)=|x|^2$
\begin{enumerate}
\item
Choose first $W_1(x)=1+|x|^2$, so that
$$LW_1(x)=\Delta W_1(x)-2x.\nabla W_1(x)=2n-2|x|^2\le -W_1(x)+(2n+1)1_{|x|^2\le 2n+1}.$$
\item Choose now $W_2(x)=e^{a|x|^2}$ for $0<a<1$,
$$LW_2(x)=(2n+4a(a-1)|x|^2)W_2(x)\le -\lambda |x|^2 W_2(x)+b1_{|x|^2\le R}$$
for some $\lambda,b,R$.
\end{enumerate}
We may now consider usual examples: let $U>c>0$ be convex (convexity and positivity outside a
large compact is sufficient if the measure is properly defined)
\begin{itemize}
\item Exponential type measures: $V(x)=U^p$ for some positive $p$.
Then,
there exists $b, c, R>0$ and $W \geq 1$ such that
$$
LW \leq - \phi(W) + b {1}_{B(0,R)}
$$
with $\phi(u)= u \, \log^{2(p-1)/p}(c+u)$ increasing. Furthermore, one can choose
$W(x)=e^{\gamma |x|^p}$ for $x$ large and $\gamma$ small enough.

\item Cauchy type measures: $V(x)=(n+\beta)\log(U)$ for some positive $\beta$.
Then,
there exists $k>2$,  $b, R>0$ and $W \geq 1$ such that
$$
LW \leq - \phi(W) + b {1}_{B(0,R)}
$$
with $\phi(u)=c u^{(k-2)/k}$ for some constant $c>0$.  Furthermore, one can choose
$W(x)=|x|^k$ for $x$ large. $k $ has to be choosen so that there exists $\epsilon>0$ such that $k+n\epsilon-2-\beta(1-\epsilon)<0$.
\end{itemize}
The details can be found in \cite{CGGR} for example.

\section{Poincar\'e's like inequalities}

The prototype of inequalities we will consider in this section is the following Poincar\'e
inequality: for every nice function $f$ there exists $C>0$ such that
$$Var_\mu(f):=\int f^2d\mu-\left(\int fd\mu\right)^2\le C\int|\nabla f|^2d\mu.$$
Poincar\'e inequalities have attracted a lot of attention due to their beautiful properties: they
are equivalent to the exponential $L_2$ decay of the associated semigroup, they give exponential
dimension free concentration of measure,... We refer to \cite{logsob,ledoux01} for historical and
mathematical references. Weak and Super Poincar\'e inequalities will be variant (weaker or
stronger) of this inequality. As we will see, this inequality may be proved very quickly using
Lyapunov conditions and local inequalities.

\subsection{Poincar\'e inequality}
Let us begin by
\begin{theorem}
Suppose that the following Lyapunov condition holds: there exists $W\ge 1$ in the domain of $L$,
$\lambda>0$, $b>0$, $R>0$ such that
\begin{equation}\label{Lpoinc}
LW\le -\lambda W+b1_{\{|x|\le R\}} \, .
\end{equation}
Assume in addition that the following local Poincar\'e inequality holds: there exists $\kappa_R$
such that for all nice functions $f$
\begin{equation}\label{localpoinc}
\int_{|x|\le R} f^2d\mu\le \kappa_R\int\Gamma(f)d\mu+\mu(\{|x|\le R\})^{-1}\left(\int_{|x|\le R} fd\mu\right)^2.
\end{equation}
Then we have the following Poincar\'e inequality: for all nice $f$
$$Var_\mu(f)\le \frac{b\kappa_R+1}{\lambda}\int\Gamma(f)d\mu.$$
\end{theorem}
As the proof is very simple, it will be quite the only one we will write completely:
\begin{proof}
Denote $c_R=\int_{|x|\le R} fd\mu$. Remark now that we may rewrite the Lyapunov condition as
$$1\le -\frac{LW}{\lambda W}+\frac{b}{\lambda}1_{\{|x|\le R\}}$$
so that by Lemma \ref{lem:philyap} and the local Poincar\'e inequality, we have
\begin{eqnarray*}
Var_{\mu}(f)&\le&\int (f-c_R)^2d\mu\\
&\le&\int (f-c_R)^2\frac{-LW}{\lambda W}d\mu+\frac{b}{\lambda}\int_{|x|\le R}(f-c_R)^2d\mu\\
&\le&\frac1\lambda\int\Gamma(f)d\mu+\frac{b}{\lambda}\kappa_R\int|\nabla f|^2d\mu
\end{eqnarray*}
which is the desired result.
\end{proof}
\begin{remark}
Using this theorem combined with the examples provided above, we may recover very simply the nice results of Bobkov \cite{bob99}
asserting that every log-concave measures (i.e. $V$ convex) satisfies a Poincar\'e inequality (\cite{BBCG}).\\
One may also easily verifies that the following two sufficient conditions for the Poincar\'e
inequality, are inherited from Lyapunov condition:
\begin{enumerate}
\item There exist $0<a<1$, $c>0$ and $R>0$ such that for all $|x|\ge R$, we have $(1-a)|\nabla
V|^2-\Delta V\ge c.$ \item There exist $c>0$ and $R>0$ such that for all $|x|\ge R$, we have
$x.\nabla V(x)\ge c|x|$.
\end{enumerate}
Note that the first one was known with $a=1/2$ for a long time but with quite harder proof.
\end{remark}

\begin{remark}
We will not develop it here but in fact quite the same may be done for Cheeger inequality ($L_1$-Poincar\'e), see \cite{BBCG}.
\end{remark}

\subsection{Weighted and weak Poincar\'e inequality}
We will now consider weaker inequalities: weighted  Poincar\'e inequalities as introduced recently
by Bobkov-Ledoux \cite{BLweight} or \cite{CGGR}, i.e. with an additional weight in the Dirichlet
form or in the variance, or weak Poincar\'e inequalities introduced by R\"ockner-Wang \cite{RW}
(see also \cite{BCR2} or \cite{CGGR}), useful to establish sub exponential concentration
inequalities or algebraic rate of decay to equilibrium for the associated Markov process. We shall
state here
\begin{theorem}
Suppose that the following $\phi$-Lyapunov condition holds: there exist some sublinear
$\phi:[1,\infty[\to\R^+$ and $W\ge 1$, $b>0$, $R>0$ such that
\begin{equation}\label{philyap}
LW\le -\phi(W)+b1_{\{|x|\le R\}}.
\end{equation}
Suppose also that $\mu$ satisfies a local Poincar\'e inequality (\ref{localpoinc}) then \begin{enumerate}
\item for all nice $f$, the following weighted Poincar\'e inequality holds
$$Var_\mu(f)\le \max\left(\frac{b\kappa_R}{\phi(1)},1\right)\int\left(1+\frac1{\phi'(W)}\right)\Gamma(f)d\mu;$$
\item for all nice $f$, the following converse weighted Poincar\'e inequality holds
$$\inf_c\int(f-c)^2\frac{\phi(W)}{W}d\mu\le (1+b\kappa_R)\int\Gamma(f)d\mu;$$
\item define $F(u)=\mu(\phi(W)<uW)$ and for $s<1$, $F^{-1}(s):=\inf\{u;F(u)>s\}$ then the following weak Poincar\'e inequality holds:
$$Var_\mu(f)\le \frac {C}{F^{-1}(s)}\int\Gamma(f)d\mu+s\, Osc_\mu(f)^2.$$
\end{enumerate}
\end{theorem}
\begin{proof}
The proof of the first two points may be easily derived using the proof for the usual Poincar\'e
inequality. For the weak Poincar\'e inequality, start with the variance, divide the integral with
respect to large or small values of $\phi(W)/W$ and use the converse Poincar\'e inequality
established previously, see details in \cite{CGGR}.
\end{proof}
\begin{remark}
Using $V(x)=1+|x|^2$ in the examples of the previous section, one gets a weighted inequality with
weight $1+|x|^2$, and converse inequality with weight $(1+|x|^2)^{-1}$ recovering results of
Bobkov-Ledoux \cite{BLweight} (with worse constants however). Note also that it enables us to get
the correct order for the weak Poincar\'e inequality (as seen in dimension 1 in \cite{BCR2}). For
this weak Poincar\'e inequality, one can find another approach in \cite{BCG} based on weak
Lyapunov-Poincar\'e inequality.
\end{remark}

\subsection{Super Poincar\'e inequality}
Our next inequality has been considered first by Wang \cite{w00} to study the essential spectrum
of Markov operators. It is also useful for concentration of measures \cite{w00} or isoperimetric
inequalities \cite{BCR3}. Wang also showed that they are, under Poincar\'e inequalities,
equivalent to $F$-Sobolev inequality (in particular one specific Super Poincar\'e inequality is
equivalent to the logarithmic Sobolev inequality), so that the results we will present now enables
us to consider a very large class of inequalities stronger than Poincar\'e inequality.
\begin{theorem}
Suppose that there is an increasing family of exhausting sets $(A_r)_{r\ge 0}$, an $r_0>0$ and a
superlinear $\phi$ such that for some $b>0$ the following Lyapunov condition holds
\begin{equation}\label{suplyap}
LW\le-\phi(W)+b1_{A_{r_0}} \, .
\end{equation}
Assume in addition that a local Super Poincar\'e inequality holds, i.e. there exists $\beta_{loc}$
decreasing in $s$ (for all $r$) such that for all $s$ and nice $f$
\begin{equation}
\label{locsup} \int_{A_r}f^2d\mu\le
s\int\Gamma(f)d\mu+\beta_{loc}(r,s)\left(\int_{A_r}|f|d\mu\right)^2 \, .
\end{equation}
Then, if $G(r):=(\inf_{A_r^c}\phi(W)/W)^{-1})$ tends to 0 as $r\to\infty$, $\mu$ satisfies for all
positive $s$
\begin{equation}\label{suppoinc}
\int f^2d\mu\le 2s\int\Gamma(f)d\mu+\tilde\beta(s)\left(\int |f|d\mu\right)^2
\end{equation}
where
$$\tilde\beta(s)=c_{r0}\beta_{loc}(G^{-1}(s),s/c_{r0})$$
and $c_{r_0}=1+b\frac{\sup_{A_{r_0}} W}{\inf_{A_{r_0}^c }\phi(W)/W}.$
\end{theorem}
\begin{proof}
In fact, one just has to play with the extra strength provided by the Lyapunov condition (i.e.
superlinear) and set $A_r$, i.e. $$\int f^2d\mu \le \int_{A_r}f^2d\mu+\int_{A_r^c}f^2d\mu\le
\int_{A_r}f^2d\mu+\frac1{\inf_{A_r^c}\phi(W)/W}\int f^2\frac{\phi(W)}{W}d\mu.$$ The first term is
treated by using the local inequality and for the second one, use the Lyapunov condition, the
crucial Lemma \ref{lem:philyap}, and once again the local Super Poincar/'e inequality. Optimize in
$r$ to get the conclusion, see details in \cite{CGWW}.
\end{proof}
\begin{remark}
Note that if the Boltzmann measure $\mu$ is locally bounded, using Nash inequality for Lebesgue
measures on balls, it is quite easy to find a local Super Poincar\'e inequality, see discussion in
\cite{CGWW}.
\end{remark}
\begin{remark}
Using this approach, one may recover famous criteria for logarithmic Sobolev inequality: convexity
Bakry-Emery criterion \cite{bakry-emery} (with worse constants), Kusuocka-Stroock conditions
\cite{KS85}, or pointwise Wang's criterion, see \cite{Wbook}.
\end{remark}

\section{Transportation's inequalities}
We will consider here another type of inequalities linking Wasserstein distance  to various
information form, namely Kullback information or Fisher information defined respectively by: if
$f$ is a density of probability with respect to $\mu$
$$H(fd\mu,d\mu):=\Ent_{\mu}(f):=\int f\log(f)d\mu$$
$$I(fd\mu,d\mu):=\int\frac{|\nabla f|^2}{f}d\mu.$$
The Wasserstein distance is defined by: for all measure $\nu$ and $\mu$
$$W_p(\nu,\mu):=\inf\left\{\E(d^p(X,Y))^{1/p};\,X\sim\nu,Y\sim\mu\right\}.$$
\subsection{Transportation and Kullback information}
Firstly, let us consider the usual transportation inequalities: for all probability density $f$
wrt $\mu$
$$W_p(\nu,\mu)\le\sqrt{ c\,H(fd\mu,d\mu)}.$$
These types of inequalities were introduced by Marton \cite{Mar96} as they imply straightforwardly
concentration of measure, and deviation inequality by a beautiful
 characterization of Bobkov-Goetze \cite{bobkov-gotze}. The case $p=1$ was proved to be equivalent to Gaussian integrability
 \cite{DGW}.

The case $p=2$ is much  more difficult:  Talagrand established the inequality for Gaussian measure
\cite{Tal96}, whereas Otto-Villani \cite{OV} and Bobkov-Gentil-Ledoux \cite{BGL} proved that a
logarithmic Sobolev inequaliy is a sufficient condition. More recently (see \cite{CatGui1}), the
authors proved that the logarithmic Sobolev inequality is strictly stronger, and provided such an
example in dimension one. We will prove here that one may give a nice Lyapunov condition to verify
this transportation inequality. Let us finish by the beautiful characterization obtained by Gozlan
\cite{Gozlan} proving that the case $p=2$ is in fact equivalent to the Gaussian dimension free
concentration of measure, see \cite{Vbook} or \cite{GLsurvey} for more on the subject. We will
prove here

\begin{theorem}
Suppose that there exists $W\ge 1$, some point $X_0$ and constants $b,c$ such that
\begin{equation}\label{t2lyap}
LW\le (-cd^2(x,x_0)+b)W
\end{equation}
then there exists $C>0$ such that for all density $f$ w.r.t. $\mu$
$$W_2(fd\mu,\mu)\le\sqrt{K\,H(fd\mu,d\mu)}.$$
\end{theorem}
\begin{proof}
Refining arguments of Bobkov-Gentil-Ledoux \cite{BGL}, the authors proved that it is in fact
sufficient to get a logarithmic Sobolev inequality for a restricted class of function, i.e.
functions $f$ such that
$$\log(f^2)\le\log\left(f^2d\mu\right)+2\eta(d^2(x,x_0)+\int d^2(x,x_0)d\mu).$$
Using truncation arguments, and mainly this class of function's property one sees how Lemma
\ref{lem:philyap} comes into play. We refer to \cite{CGW} for the tedious technical details.
\end{proof}

\begin{remark}
It is not difficult to remark that for $V(x)=x^3+3x^2\sin(x)+x$ near infinity, the Lyapunov
condition is verified. However the logarithmic Sobolev inequality does not hold in this case as
shown in \cite{CatGui1}.
\end{remark}

\subsection{Transportation and Fisher Information}
Transportation-information inequalities with Fisher inequalities were  only very recently studied
in \cite{GLWY,GLWW,GJLW}, because of their equivalence with deviation inequality for Markov
processes due to large deviations estimation. The two main interesting ones are for $p=1$ and
$p=2$: for all probability density $f$ w.r.t. $\mu$
$$W_p(fd\mu,d\mu)\le\sqrt{C\, I(fd\mu,d\mu)}.$$
In \cite{GLWY}, various criteria were studied, such as Lipschitz spectral gap. In particular if
$p=1$ and the distance is the trivial one, this i,equality is in fact equivalent to a Poincar\'e
inequality. These authors  also proved:
\begin{theorem}
Suppose that a Poincar\'e inequality holds, and that the following Lyapunov condition holds: there exists $W\ge 1$,$x_0$, $c,b>0$ such that
\begin{equation}
\label{wilyap} LW\le -cd^2(x,x0)W+b \, .
\end{equation}
Then we have for all probability density $f$ w.r.t. $\mu$
$$W_1(fd\mu,d\mu)\le\sqrt{C\, I(fd\mu,d\mu)}.$$
\end{theorem}
\begin{proof}
Let us scheme the proof: by \cite{Vbook}
\begin{eqnarray*}
W_1(fd\mu,d\mu)&\le&\int d(x,x_0)|f-1|d\mu\\
&\le&\sqrt{\int|f-1|d\mu}\sqrt{\int d^2(x,x_0)|f-1|d\mu}.
\end{eqnarray*}
For the first term, we use the fact that the Poincar\'e inequality is equivalent to a control of the total variation by the square of the Fisher Information, and for the second one the Lyapunov condition. One has of course to be careful as $|f-1|$ is not in the domain of $L$, so that an approximation argument has to be done. We refer to \cite{GLWY} for details.
\end{proof}

\begin{remark}
It is quite easy to remark that a logarithmic Sobolev inequality implies a transportation information inequality with Fisher information in the case $p=2$, but it is unknown if it is strictly weaker. {\it A fortiori}, no Lyapunov condition is known in the case $p=2$.
\end{remark}

\section{Logarithmic Sobolev inequalities under curvature}
Recall the classical Logarithmic Sobolev inequality, i.e. for all nice $f$
$$Ent_\mu(f^2)\le c\int\Gamma(f)d\mu \, .$$ This inequality has a long history.

Initiated by Gross \cite{gross} to study hypercontractivity, it was largely studied by many
authors due to its relationship with the study of decay to equilibrium, concentration of measure
property, efficacity in spin systems study, see \cite{bakry,logsob,ledoux01,Wbook,Vbook} for
further references. A breakthrough condition was the Bakry-Emery one: namely if
$Hess(V)+Ric\ge\delta>0$ then a logarithmic Sobolev holds. Kusuocka-Stroock gave a Lyapunov-type
condition (recovered by the study given in the Super-Poincar\'e case), and using Harnack
inequalities, Wang proved that in the lower bounded curvature case, i.e.
\begin{equation}\label{curv}
Hess(V)+Ric\ge\delta
\end{equation}
with $\delta$ maybe negative, a sufficient Gaussian integrability, i.e. $\int e^{((-\delta)_+/2+\epsilon)|x|^2}d\mu<\infty$, is enough
to prove a logarithmic Sobolev inequality. We will prove here
\begin{theorem}
Suppose that (\ref{t2lyap}) and (\ref{curv}) hold then $\mu$ satisfies a logarithmic Sobolev inequality.
\end{theorem}

\begin{proof}
Remark first that by (\ref{t2lyap}), a Poincar\'e inequality holds due to the effort of Section 3. Remark also that by Lyapunov conditions
(and maximization argument)
$$W_2^2(fd\mu,d\mu)\le 2\int d^2(x,x_0)|f-1|d\mu\le I(fd\mu,d\mu)+C.$$
Use now a HWI inequality of Otto-Villani \cite{OV} (see also \cite{BGL}):
$$H(fd\mu,d\mu)\le2\sqrt{I(fd\mu,d\mu)}W_2(fd\mu,d\mu)-\frac\delta2W_2^2(fd\mu,d\mu).$$
so that a defective logarithmic Sobolev inequality holds, that may be tightened via Rothaus' Lemma due to the Poincar\'e inequality. We refer to \cite{CGW} for details.
\end{proof}
\begin{remark}
Let us give here an example not covered by Wang's condition. On $\R^2$, take $V(x,y)=r^2g(\theta)$ in polar coordinates with $g(\theta)=2+\sin(k\theta)$. It is not hard to remark that $Hess(V)$ is bounded and that the Lyapunov condition (\ref{t2lyap}) is verified. However Wang's integrability condition is not verified, despite the fact that a logarithmic Sobolev inequality do hold by our theorem.
\end{remark}

\begin{remark}
One may also give Lyapunov conditions when $\delta$ is replaced by some unbounded function of the distance.
\end{remark}

\medskip

{\bf Acknowledgements:} A.G. thanks the organizers of this beautiful Grenoble Summer School and wonderful conference.

\bibliographystyle{plain}
\bibliography{cg-procgrenoble}

\end{document}